\documentclass{amsart}

\title{Research Statement}
\author{xxxxxx}

\usepackage{amsfonts, amssymb, amsmath, amscd, xypic, latexsym, epsfig, mathrsfs, yfonts}
\usepackage{upref}
\usepackage[all]{xy}
\input xy
\xyoption{all}

\usepackage{hyperref}
\newcommand{\tw}{.4\textwidth}

\newtheorem{thm}{Theorem}[subsection]

\newtheorem{defn}{Definition}[subsection]
\newtheorem{rem}[thm]{Remark}

\newtheorem{defi}{Definition}[section]

\newtheorem{lemm}{Lemma}[section]
\newtheorem{rema}{Remark}[section]
\newtheorem{theorem}{Theorem}[section]
\begin{document}
\begin{center}
\textbf{\Large A new geometric approach to problems in birational geometry}

{\ }\\
\vspace{0.5cm}
\textbf{Chen-Yu Chi and Shing-Tung Yau}
\end{center}
\section{Introduction}
In this paper, we initiate a program to study problems in
birational geometry. This approach will be more geometric than
other more algebraic approaches. Most of the arguments can,
however, be phrased in a purely algebraic way. It is quite likely
some of them can be applied to deal with the geometry over
different ground fields.

Given a projective variety $M$, we shall study the geometric
information provided by the pluricanonical space $H^0(M, mK_M)$.
Note that the minimal model program, as led by Mori, Kawamata,
Koll\'ar and others, has achieved great success. While the earlier
workers had solved the problem for threefolds completely, the
spectacular finite generation question was recently solved by
several people, using different approaches: the analytic approach
due to Siu \cite{siu2} and the algebraic approach due to Birkar,
Cascini, Hacon and M\textsuperscript{c}Kernan \cite{bchm}. In our
approach, instead of using the full canonical ring, we shall focus
our study on the pluricanonical space for a fixed $m$.

Ideally, we would like to determine the birational type of our
algebraic variety based on the information on this space only. Any
birational transformations of algebraic manifolds will induce a
linear map between the corresponding pluricanonical spaces (for
each fixed $m$). The plurigenera are of course invariant under the
birational transformations. But more importantly, there are other
finer invariants that are preserved by these transformations. The
most important ones are the natural normlike functions (called
norms in this introduction) induced by integrating over $M$ the
$m$-th root of the product of a $m$-pluricanonical form and its
conjugate. The norm defines an interesting geometry which was not
explored extensively before.

In our work, we shall initiate a program to study this geometry.
The first major questions we address to are the following ones:
\smallskip

{\bf 1.~Torelli type theorem.} Given two algebraic varieties $M$
and $M'$, suppose there is a linear map that defines an isometry
(with respect to the norm mentioned above) between the two  normed
vector spaces $H ^0(M, mK_M)$ and $H^0(M', mK_{M'})$. We claim
that with a few exceptional cases of $M$ and $M'$, the linear
isometry is induced by a birational map between $M$ and $M'$. This
can be considered as a Torelli type theorem in birational
geometry.

We call this kind of theorem a Torelli type theorem because the
classical Torelli theorem says that the periods of integrals
determine an algebraic curve. This remarkable theorem was
generalized to higher dimensional algebraic varieties. The most
notable one was the work of Piatetsky-Shapiro and Shafarevich
\cite{p-s--s} for algebraic K3 surfaces, which was generalized to
K\"ahler K3s by Burns-Rapoport \cite{b-r}, where they proved the
injectivity of period maps. The surjectivity of period maps for K3
surfaces was done using Ricci flat metrics by Siu \cite{siu1} and
Todorov \cite{todorov} following the work of Kulikov \cite{Kuli}
and of Perrson and Pinkham \cite{p-p}. This phenomena of
surjectivity is known to be rather generic, and in many cases the
period map can be proved to have degree one for hypersurfaces (see
e.g. Donagi \cite{donagi}).

{\bf 2.~Existence.} Characterize geometrically and algebraically
those normed vector spaces that can be realized as the
pluricanonical spaces of some algebraic varieties the way above.
Hopefully, there may be some effective way to construct the
birational models of these varieties.

{\bf 3.~Computation.} In the case of the classical Torelli
theorems, the periods can be effectively computed by methods dated
back to Picard, Leray, Dwork and others. We hope to calculate
these normed spaces effectively too. Some differential geometric
methods will be brought in.

{\bf 4.~Relations with questions of GIT and other invariants.}
Making use of the pluricanonical series, we are able to form new
invariant (pseudo-)metrics on the algebraic manifolds. There
should be some relationship between these metrics and other well
known canonical metrics such as K\"ahler-Einstein metrics. We hope
to build up a link between our approach with other metrical
approaches to algebraic geometry.\smallskip

In this paper, we shall prove that when $|mK_M|$ has no base point
and defines a birational map, the normed space is indeed powerful
enough to determine the birational type of the algebraic
varieties. We can achieve this when $m$ is large enough (depending
on the dimension of $M$ only). Indeed we prove a Torelli theorem
that is described in {\bf 1} under rather general assumptions. We
should say that in case the manifold is one dimensional and $m=2$,
the problem was treated by Royden in his study of the
biholomorphic transformations of Teichm\"uller space. We think it
is possible to generalize Royden's work to higher dimensional
manifolds.

We shall study {\bf 2} by using a more differential geometric
approach. We here outline what kind of metric we can obtain. At
every point $\eta_0\in H^0(M, mK_M)$ and $\eta_1,\eta_2\in H^0(M,
mK_M)$, viewed as two tangent vectors at $\eta_0$ in $H^0(M,
mK_M)$, we define a hermitian metric
$$
h(\eta_1,\eta_2)=\int_M\frac{\eta_1}{\eta_0}
\overline{\left(\frac{\eta_2}{\eta_0}\right)}\langle\eta_0\rangle_m
$$
where $\frac{\eta_1}{\eta_0}$ and $\frac{\eta_2}{\eta_0}$ are
viewed as meromorphic functions on $M$ and
$\langle\eta_0\rangle_m$ is a real nonnegative continuous
$(n,n)$-form as defined in {\bf 2.1}. $H^0(M, mK_M)$ is then given
with the structure of a hermitian manifold. This hermitian
structure is closely related to the norm we mentioned above.
Actually the norm function on $H^0(M, mK_M)$ is the K\"ahler
potential of this hermitian metric in a suitable sense. More
details will be given in \cite{cy}.

\section{Pseudonorms on $H^0(M,mK_M)$ and their asymptotic properties}
\subsection{The pseudonorm $\langle\langle\ \rangle\rangle_m$}
Let $M$ be a complex manifold of dimension $n$. To every $\eta\in
H^0(M,mK_M)$ we can associate a real nonnegative continuous
$(n,n)$-form on $M$, denoted as $\langle\eta\rangle_m$, as
follows:

let $\mathcal{U}=\{(U, (w^j_U=u^j_U+iv^j_U)^n_{j=1})\}$ be an open
cover of $M$ of coordinate charts. If $\eta_{|_U}=\eta_U
(dw^1_U\wedge\dots\wedge dw^n_U)^{\otimes m}$ with
$\eta_U\in\mathcal{O}_M(U)$, we can define on $U$ a real
nonnegative continuous $(n,n)$-form
$$
\langle\eta_{|_U}\rangle_m=|\eta_U|^{\frac{2}{m}}du^1_U\wedge
dv^1_U\dots\wedge du^n_U\wedge dv^n_U
$$
and can verify that
$\{\langle\eta_{|_U}\rangle_m\}_{U\in\mathcal{U}}$ does give a
globally defined form, denoted as $\langle\eta\rangle_m$. It is
routine to see that this definition does not depend on the choice
of $\mathcal{U}$.

If $M$ is compact, we define
$$\langle\langle\eta\rangle\rangle_m=\int_M \langle\eta\rangle_m$$ and will abbreviate
it as $\langle\langle\eta\rangle\rangle$ if $m$ is clear in the context.\
Therefore, for a compact complex manifold $M$ we have defined a function
$$\langle\langle\ \rangle\rangle : H^0(M,mK_M)\to\mathbf R_{\geq 0}$$
and will call it the $pseudonorm$ associated to $mK_M$.

From the fact that $|a+b|^{\alpha}\leq|a|^{\alpha}+|b|^{\alpha}$
for any $0< \alpha <1$ and $a, b\in\mathbf C$ we can verify the
triangle inequality
$\langle\langle\eta_1+\eta_2\rangle\rangle\leq\langle\langle\eta_1\rangle\rangle
+\langle\langle\eta_2\rangle\rangle$ for any $\eta_1,\eta_2\in
H^0(M,mK_M)$. From the definition
$\langle\langle\eta\rangle\rangle=0$ if and only if $\eta=0\in
H^0(M,mK_M)$. However, $\langle\langle
c\eta\rangle\rangle=|c|^{\frac{2}{m}}\langle\langle\eta\rangle\rangle$
for $c\in\mathbf C$, which shows that $\langle\langle\
\rangle\rangle$ is not a norm if $m\neq 2$.\

We define a metric space structure on $H^0(M,mK_M)$ using
$\langle\langle\ \rangle\rangle$ by
$$
d(\eta_1,\eta_2)=\langle\langle\eta_1-\eta_2\rangle\rangle \text{
for any}\ \eta_1,\eta_2\in H^0(M,mK_M)\text{.}
$$
$H^0(M,mK_M)$ so metrized will be denoted as $\big(H^0(M,mK_M),
\langle\langle\ \rangle\rangle\big)$.

If $\varphi :M'\dashrightarrow M$ is a birational map, then the
induced isomorphism $\Phi : \big(H^0(M,mK_M),\langle\langle\
\rangle\rangle\big)\to \big(H^0(M',mK_{M'}),\langle\langle\
\rangle\rangle\big)$ is an isometry.

\subsection{A local asymptotic expansion}
We will state the main local asymptotic result, whose proof can be found in \cite{chi}, and then deduce from it the global one in the next section, namely the asymptotic property of $\langle\langle\ \rangle\rangle$.

We first settle the notation as follows:\\
\\
$n,m\in\mathbf N$, $m>2$, $\Delta_0=\{(z_1,\dots ,z_n)\in\mathbf C^n|\ |z_j|<1,j=1\dots ,n\}$,\\
\\
$\chi (x_1,y_1,\dots, x_n, y_n)\in\mathcal C^{\infty}(\overline\Delta_0)$, $\phi (z_1,\dots ,z_n)\in\mathcal O (\overline\Delta_0)$,\\
\\
$A=(a_1,\dots ,a_n)$, $B=(b_1,\dots ,b_n)\in (\mathbf N\cup\{0\})^n$,\\
\\
$l_j=\frac{b_j+1}{a_j}$ if $a_j\neq 0$ and $=\infty$ otherwise, $j=1,\dots ,n$.\\
\\
$l=\min{\{l_j|\ j=1,\dots ,n\}}$,\\
\\
Assume that $l(A,B)=l_1=\dots =l_{\mu (A,B)}<l_{\mu (A,B)+1}\leq\dots\leq\l_n$. Notice that $l(A,B)$ and $\mu (A,B)$ only depend on the multi-indices $A$ and $B$. If $A$ and $B$ are clear in our arguments we will denote $l(A,B)$ and $\mu (A,B)$ by $l$ and $\mu$ respectively.

We abbreviate $(x_1,y_1,\dots, x_n, y_n)$, $(z_1,\dots ,z_n)$, $z^{a_1}_1\dots z^{a_n}_n$, $|z_1|^{b_1}\dots |z_n|^{b_n}$, and $dx_1 dy_1 \dots dx_n dy_n$
as $(X,Y)$, $Z$, $Z^A$, $|Z|^B$, and $dX\ dY$ respectively. Let
$$\Psi (t)=\int_{\overline\Delta_0}\chi
(X,Y)\big|Z^A+t\phi (Z)\big|^{\frac{2}{m}}\big|Z\big|^{2B} dXdY\text{.}$$
\begin{thm}\ \\
$$\Psi (t)-\Psi (0)=\left\{
\begin{array}{cr}
O\left(|t|\big(ln\frac{1}{|t|}\big)^{\mu}\right) & \text{if}\ \ 2l+\frac{2}{m}\geq 1;\\
\\
\begin{array}{l}
c(A,B,\phi )\ |t|^{2l+\frac{2}{m}}\big(ln\frac{1}{|t|}\big)^{\mu -1}\\
+o\left(|t|^{2l+\frac{2}{m}}\big(ln\frac{1}{|t|}\big)^{\mu -1}\right)\end{array} & \text{if}\ \ 2l+\frac{2}{m}<1\text{,}
\end{array} \right.$$
where $c(A,B,\phi)$ is a real number depending on $\phi$. In the last case we have $c(A,B,\phi)\geq 0$, and $$c(A,B,\phi)=0\ \Longleftrightarrow\ \phi (0,\dots ,0, z_{\mu +1},\dots , z_n)\equiv 0\text{.}$$
\end{thm}

\begin{rem}
More specific information when $2l+\frac{2}{m}\geq 1$ can be given.
We actually can show that when $k<2l+\frac{2}{m}<k+1$ where $k\in\mathbf{N}$, $\Psi^{(j)}(0)$ exists for $j\leq k$.
In addition, we get in this case an asymptotic expansion for $\Psi (t)-\stackrel{k}{\sum\limits_{j=0}}\frac{\Psi^{j}(0)}{j!}t^j$.
We will not need these in this paper so we omit them here. For more detail see \cite{chi}.
\end{rem}

In $\mathbf{2.4}$ we will apply this result to obtain the main global result. Let
$\eta_0,\eta\in H^0(M,mK_M)$. We hope to describe the asymptotic behavior of
$\langle\langle\eta_0 + t\eta\rangle\rangle$ as $t\to 0$. \\

\subsection{The characteristic index and indicatrix}
Before getting into the deduction of the global asymptotic
expansion, we introduce several quantities measuring how singular
a divisor is at a point in the ambient space.

Let $M$ be a smooth variety, $D$ a nonzero effective divisor on $M$.
(In $\mathbf{2.4}$ $D$ will be chosen to be $\{\eta_0=0\}$ for the $\eta_0\in H^0(M, mK_M)$ we consider.)
We first choose a log resolution
$\pi :\widetilde M\to M$ for the pair $(M,D)$ and write $$\pi^*D=\sum\limits_{E}a_E E\quad\text{and}\quad K_{\widetilde M}=\pi^* K_M+ \sum\limits_{E}b_E E\text{,}$$
where $E$ runs over all irreducible subvarieties of $\widetilde M$ of codimension $1$.
\begin{defn}
$(1)$ For every $x\in M$ the $log\ canonical\ threshold$ of $D$ at $x$, denoted as ${\rm lct}(D,x)$, is given by
$${\rm lct}(D,x)=\min\limits_{\{E|x\in\pi (E)\}}\frac{b_E +1}{a_E}\text{.}$$ We also have the global log canonical threshold of $D$,
$${\rm lct}(D)=\min\limits_E\ \frac{b_E +1}{a_E}\text{.}$$ It is clear that ${\rm lct}(D)=\min\limits_{x\in M}\ {\rm lct}(D,x)$.\\

$(2)$ For every $x\in M$ the $log\ canonical\ multiplicity$ of $D$ at $x$, denoted as $\mu (D,x)$, is given by
$$\mu (D,x)=\max\left\{\ q\ \left|
\begin{array}{l}
\text{There exist distinct irreducible divisors}\ E_1,\dots,E_q\ \text{in}\ \widetilde M\ \\
\text{such that}\ \frac{b_{E_j}+1}{a_{E_j}}={\rm lct}(D,x)\ \text{for all}\ j\ \text{and}\ x\in\pi (\cap E_j)\text{.}\\
\end{array}\right.\right\}\text{.}$$\\

$(3)$ The $characteristic\ index$ of $D$ at $x$ is the pair $({\rm lct}(D,x),\mu (D,x))$. Consider the following total order:$$(l_1,\mu_1)>(l_2,\mu_2)\Longleftrightarrow\left\{
\begin{array}{c}
l_1=l_2\ \text{and}\ \mu_1>\mu_2\\
\text{or}\\
l_1<l_2\\
\end{array}\right.$$ \\
The $global\ characteristic\ index$ of $D$, denoted as $\big({\rm lct}(D),\mu (D)\big)$, is given by
$$\big({\rm lct}(D),\mu (D)\big)=\sup\limits_{p\in M} \big({\rm lct}(D,p),\mu (D,p)\big)\text{.}$$\\

$(4)$ We define the $characteristic\ indicatrix$ of $(M,D)$,
denoted as $C(D)$, to be the set of points achieving global
characteristic index, i.e. $$C(D)=\left\{x\in M\left|\ \big({\rm
lct}(D,x),\mu (D,x)\big)=\big({\rm lct}(D),\mu
(D)\big)\right.\right\}\text{.}$$ Notice that in general $C(D)$ is
different from the minimal log canonical centers of $D$.
\end{defn}

The total order defined here is adapted to the comparison of the
asymptotic order of functions of the form
$|t|^l\big(ln\frac{1}{|t|}\big)^{\mu}$. We have
\begin{align}
|t|^{l'}\left(ln\frac{1}{|t|}\right)^{\mu'}=o\left(|t|^l\big(ln\frac{1}{|t|}\big)^{\mu}\right)\ \text{if}\ (l,\mu)>(l',\mu')\text{.}
\end{align}
Let $\mathcal E$ be the set of all irreducible divisors $E$ such
that $\frac{b_{E} +1}{a_E}={\rm lct}(D)$,
$$\widetilde M_{D,r}=\bigsqcup\limits_{
E_1,\ldots, E_r:\ \text{distinct in}\ \mathcal{E} }\
E_1\cap\dots\cap E_r\text{,}$$ and $\iota_r:\widetilde
M_{D,r}\to\widetilde M$ the canonical morphisms induced by
inclusions. We have
\begin{align}
\mu (D)=\sup\{r|\ \widetilde M_{D,r}\neq\phi\}\ \text{and}\ C(D)=\pi\iota_{\mu (D)}(\widetilde M_{D,\mu (D)}).
\end{align}

The log canonical thresholds is well defined, namely it is
independent of the choice of log resolutions. In fact
$$
{\rm lct}(D,x)=\inf\{c>0|\mathcal J(M,cD)_x\neq \mathcal O_{M,x}\}\text{,}
$$
and the multiplier ideal sheaves $\mathcal J(M,cD)$ do not depend
on the log resolution we choose.

This also gives the basic inequality
\begin{align}
{\rm lct}(D,x)\leq\frac{n}{{\rm mult}_x D}
\end{align}
(by taking a blow-up $Bl_x(M)\to M$ followed by a log resolution).

\begin{rem}
In the rest of the paper we do not need $\mu (D,x)$ and $C(D)$ to
be independent of the choice of log resolution. For each divisor
$D$ we can simply choose a fixed resolution to define $\mu (D,x)$
and $C(D)$. However, they can indeed be defined in terms of some
resolution free ideal sheaves, hence are both independent of the
choice of resolution (see \cite{chi}). Instead of giving a formal
proof of this independence here, we would like to point out that
one can see this, at least analytically, by using Theorem 2.2.1
for the case $\chi=\phi\equiv 1$ and that pulling back a
differential form by an analytic modification does not change its
integral.
\end{rem}

\begin{rem}
The characteristic index is a finer measurement of singularity
than ${\rm lct}$ is. For example, ${\rm lct}$ alone can not tell
between a reduced non-smooth s.n.c.\ divisor and a smooth divisor.
Higher characteristic indices correspond to worse singularities.
In this sense, the characteristic indicatrix $C(D)$ is the set of
points at which the pair $(M, D)$ is the most singular.
\end{rem}

\subsection{The asymptotic property of $\langle\langle\ \rangle\rangle_m$}
In this subsection we assume $M$ to be compact. We first give the local setting.
As in $\mathbf{2.1}$, let $\mathcal{U}=\{(U, (w^j_U)^n_{j=1})\}$ be a finite open
cover of coordinate charts on $M$. We choose a log resolution
$\pi: \widetilde M\to M$ for $(M,D_{\eta_0}=\{\eta_0=0\})$ and a finite
refinement $\mathcal{V}=\{(V, Z_V=X_V+iY_V)\}$ of $\pi^{-1}\mathcal{U}=\{\pi^{-1}U\}$
formed by charts in $\widetilde M$, where $Z_V$ and $(X_V,Y_V)$ abbreviate
$(z^1_V,\dots ,z^n_V)$ and $(x^1_V,y^1_V\dots ,x^n_V,y^n_V)$ respectively. Let $\tau : \mathcal V\to\mathcal U$ be such that $\pi(V)\subset\tau (V)$. Finally, we choose a partition
of unity $\{\chi_V(X_V,Y_V)\}$ subordinate to $\mathcal{V}$. $\mathcal V$ and $\{\chi_V\}$ can be so chosen that\\
\\
(i) the image of $Z_V:V\to\mathbf C^n$ is
$\Delta_0=\{(z_1,\dots ,z_n)\in\mathbf C^n|\ |z_j|<1\ \text{for all}\ j\}$;\\
\\
(ii) if $U=\tau (V)$, then
$$\pi^*(dw^1_U\wedge\dots\wedge dw^n_U)=\big(j_V(Z_V)\big)Z^{B_V}_V\ dz^1_V\wedge\dots\wedge dz^n_V$$ for some nonvanishing $j_V\in\mathcal O (\overline\Delta_0)$ (hence all its derivatives are bounded) and multi-index
$B_V=(b^1_V,\dots ,b^n_V)\in (\mathbf N\cup{0})^n$;\\
\\
(iii) following the notation in (ii), we have $$\pi^*\eta_0=c_V(Z_V)\big(j_V(Z_V)\big)^mZ^{A_V+mB_V}_V (dz^1_V\wedge\dots\wedge dz^n_V)^{\otimes m}$$ and
$$\pi^*\eta=c_V(Z_V)\big(j_V(Z_V)\big)^m\phi_V(Z_V)Z^{mB_V}_V(dz^1_V\wedge\dots\wedge dz^n_V)^{\otimes m}$$
where $\phi_V$ and $c_V\in\mathcal O (\overline\Delta_0)$, $c_V$ is nonvanishing, and $A_V=(a^1_V,\dots ,a^n_V)\in (\mathbf N\cup{0})^n$;\\
\\
(iv) for each $V$ we have $l^1_V=\dots =l^{\mu_V}_V<l^{\mu_V+1}_V\leq\dots\leq l^n_V$, where $l^j_V=\frac{b^j_V+1}{a^j_V}$.\\
\\
(v) $\chi_V (0,0)\neq 0$ for every $V$.

\begin{rem}
In (ii) and (iii) and in the following proof of Theorem 2.4.2 we
will have to consider two different kinds of pullbacks via
$\pi:\widetilde M\to M$ of elements in $H^0(M, K_M)$, and it is
important not to mix them up. The first one is $\pi^*:H^0(M,
mK_M)\to H^0(\widetilde M, mK_{\widetilde M})$ which acts on $K_M$
as the usual pullback of differential forms via the map $\pi$. The
second one is $\pi^{**}:H^0(M, mK_M)\to H^0\big(\widetilde M,
\pi^*(mK_M)\big)$, the usual pullback map from the sections of a
vector bundle to those of its pullback bundle via a map.
\end{rem}

In terms of the $\mathcal V$ and ${\chi_V}$ chosen above we can write
$$\langle\langle\eta_0 + t\eta\rangle\rangle=
\sum\limits_{V\in\mathcal V}\int_{\overline\Delta_0} \big(\chi_V
|c_V|^{\frac{2}{m}}|j_V|^2\big)\big|Z^{A_V}_V +
t\phi_V\big|^{\frac{2}{m}}\big|Z_V\big|^{2B_V}dX_VdY_V\text{.}$$
Our main asymptotic result for $\langle\langle\ \rangle\rangle_m$
is the following
\begin{thm}
Given $\eta_0,\eta\in H^0(M,mK_M)$, let $(l,\mu )=\big({\rm lct}(D_{\eta_0}),\mu (D_{\eta_0})\big)$ and $C(D_{\eta_0})$ be defined as in $\mathbf{2.3}$. We have\\
$$\ \langle\langle\eta_0 + t\eta\rangle\rangle -\langle\langle\eta_0\rangle\rangle=
\left\{
\begin{array}{cr}
O\left(|t|\big(ln\frac{1}{|t|}\big)^{\mu}\right) & \text{if}\ \ 2l+\frac{2}{m}\geq 1\text{;}\\
\\
\begin{array}{l}
c(\eta_0,\eta)\ |t|^{2l+\frac{2}{m}}\big(ln\frac{1}{|t|}\big)^{\mu -1}\\
+o\left(|t|^{2l+\frac{2}{m}}\big(ln\frac{1}{|t|}\big)^{\mu  -1}\right)\end{array} & \text{if}\ \ 2l+\frac{2}{m}<1\text{,}
\end{array} \right.
$$
where $c(\eta_0,\eta)$ is a real number depending on $\eta_0$ and $\eta$. In the last case we have $c(\eta_0,\eta)\geq 0$, and $$c(\eta_0,\eta)=0\ \Longleftrightarrow\ \eta\ \text{vanishes on}\ C(D_{\eta_0})\text{.}$$
\end{thm}

\begin{proof} Following the notation at the beginning of $\mathbf{2.2}$, for each $V\in\mathcal{V}$ we obtain correspondingly a pair
$(l_V,\mu_V)$. It is clear that $(l,\mu)=\sup\limits_V (l_V,\mu_V)$ according to the total order we introduced in $\mathbf{2.3}$(3). For each $V$, applying Theorem 2.2.1 to the case $\chi=\chi_V |c_V|^{\frac{2}{m}}|j_V|^2$, $\phi=\phi_V$, $A=A_V$ and $B=B_V$, and then summing up the corresponding asymptotic expansions, we obtain the expected expansion.

For the statement about $c(\eta_0,\eta)$, notice that, by (2.1), only those $V$ with $(l_V,\mu_V)=(l,\mu)$ will
contribute to $c(\eta_0,\eta)$.
More precisely,
$$c(\eta_0,\eta)=\sum\limits_{\{V|(l_V,\mu_V)=(l,\mu)\}}c(A_V,B_V,\phi_V)\text{.}$$
By Theorem 2.2.1 we know that $c(\eta_0,\eta)\geq 0$ and
$$\begin{array}{l}
c(\eta_0,\eta)=0\\
\\
\Leftrightarrow
\begin{array}{l}
c(A_V,B_V,\phi_V)=0\ \text{for all}\ V\\
\text{such that}\ (l_V,\mu_V)=(l,\mu)\\
\end{array}\\
\\
\Leftrightarrow
\begin{array}{l}
\phi_V(0,\dots ,0,z^{\mu +1}_V,\dots ,z^n_V)\equiv 0\\
\text{for all}\ V\ \text{such that}\ (l_V,\mu_V)=(l,\mu)\text{.}
\end{array}
\end{array}$$
We know that $\iota_{\mu}(\widetilde M_{D,\mu})$ (see (2.2)) is defined by $z^1_V=\dots =z^{\mu}_V=0$ in every such $V$. Regarding the conditions (ii) and (iii) above satisfied by the $\mathcal V$ we choose, the last statement is equivalent to saying that $p^{**}\eta$ vanishes on $\iota_{\mu}(\widetilde M_{D,\mu})$. This is the same as saying that $\eta$ vanishes on $\pi\iota_{\mu}(\widetilde M_{D,\mu})=C(D_{\eta_0})$.
\end{proof}
\section{Identifying the Images of Rational Maps $\varphi_{|mK_M|}$}
We still assume $M$ to be compact. In this section we are going to use Theorem 2.4.2 to study the image of
the rational map $\varphi=\varphi_{|mK_M|}$ associated to the linear system $|mK_M|$.

Let $B=Bs|mK_M|$. First we recall the definition of $\varphi$. It is given by
$$
\begin{array}{cccl}
\varphi: & M & \dashrightarrow & \mathbb{P}H^0(M,mK_M)^*\\
\\
& x & \longmapsto &
\begin{array}{c}
\big\{\eta\in H^0(M,mK_M)\ \big|\ \eta (x)=0\big\}\ \text{viewed}\\
\text{as a hyperplane of}\ H^0(M,mK_M)\text{.}
\end{array}
\end{array}
$$
Notice that $\varphi$ is defined only for $x\in M-B$. (Otherwise $\{\eta|\ \eta (x)=0\}=H^0(M,mK_M)$ is not a hyperplane.)
In general, for any hyperplane $H$ in $H^0(M,mK_M)$ we have
$$
H\stackrel{\alpha}{\subseteq}\{\eta|\ \eta_{|_{Bs|H|}}\equiv 0\}=\{\eta|\ \eta_{|_{Bs|H|-B}}\equiv 0\}
\stackrel{\beta}{\subseteq} H^0(M, mK_M)$$
and $Bs|H|-B=\varphi^{-1}(H)$.
Therefore
\begin{align}
\begin{array}{l}
H\ \text{is in the image of}\ \varphi\Longleftrightarrow Bs|H|-B\neq\phi\Longleftrightarrow\beta\ \text{is}\subsetneq\ \\
\\
\Longleftrightarrow\ \alpha\ \text{is an equality}\ \Longleftrightarrow H=\{\eta|\ \eta_{|_{Bs|H|-B}}\equiv 0\}\text{.}
\end{array}
\end{align}

\noindent {\bf Question:} Given $H$ in the image of $\varphi$, can
we characterize $H$ by a subset of the hyperplane in $H^0(M,mK_M)$
it represents and metrical properties of $\langle\langle\
\rangle\rangle$? \smallskip

\noindent {\bf Idea:} If we can find $\eta_0\in H^0(M,mK_M)$ such
that $2{\rm lct}(D_{\eta_0})+\frac{2}{m}<1$ and $\phi\neq
C(D_{\eta_0})-B\subseteq\ Bs|H|-B$, then
$$
\begin{array}{l}
H\stackrel{(3.1)}{=}\{\eta|\ \eta_{|_{Bs|H|-B}}\equiv 0\}\subseteq\{\eta|\ \eta_{|_{C(D_{\eta_0})-B}}\equiv 0\}\\
\\
\subsetneq H^0(M, mK_M)\ \text{since}\ C(D_{\eta_0})-B\neq\phi\text{,}
\end{array}$$ hence the $\subseteq$ above is actually an equality, and by Theorem 2.4.2
$$H=\{\eta|\ \eta_{|_{C(D_{\eta_0})-B}}\equiv 0\}=\{\eta|\ \eta_{|_{C(D_{\eta_0})}}\equiv 0\}=\{\eta|\ c(\eta_0,\eta)=0\}\text{.}$$
We know that $c(\cdot,\cdot)$ can be read off from $\langle\langle\ \rangle\rangle$.\\
\begin{defi}
We say that property $\rm (CS)$ (standing for "concentrating singularities") holds for $mK_M$ if for a generic $H$ in the image of $\varphi$
there exists $\eta_0\in H^0(M,mK_M)$ such that $2{\rm lct}(D_{\eta_0})+\frac{2}{m}<1$ and $\phi\neq C(D_{\eta_0})-B\subseteq Bs|H|-B$.
\end{defi}

The following is the main ingredient in using metrical properties of pseudonorms to identify images of rational maps of the form we consider above.

\begin{lemm}
Let $M$, $M'$ be compact complex manifolds. If $\rm (CS)$ holds for both $mK_M$ and $mK_{M'}$ and
$$\iota:\big(H^0(M,mK_M),\langle\langle\ \rangle\rangle_m\big)\to\big(H^0(M',mK_{M'}),\langle\langle\ \rangle\rangle_m\big)$$
is a linear isometry, then the isomorphism induced by $\iota$,
$$I:\mathbb{P}H^0(M,mK_M)^*\to\mathbb{P}H^0(M',mK_{M'})^*\text{,}$$
maps the closure of the image of $\varphi_{|mK_M|}$ isomorphically onto that of $\varphi_{|mK_{M'}|}$.
\end{lemm}

\begin{proof}
By symmetry, it suffices to prove that $I$ maps a generic point in
the image of $\varphi_{|mK_M|}$ into that of
$\varphi_{|mK_{M'}|}$.

By $\rm (CS)$, for a generic $H$ in the image of $\varphi$ we
select a section $\eta_0\in H^0(M,mK_M)$ such that $2{\rm
lct}(D_{\eta_0})+\frac{2}{m}<1$ and $\phi\neq
C(D_{\eta_0})-B=Bs|H|-B$. We already know from $\mathbf{Idea}$
that $H=\left\{\eta\in H^0(M,mK_M)\big|\
c(\eta_0,\eta)=0\right\}$.

By the definition of $I$ and the fact that $\iota$ is a linear
isometry,
$$\begin{array}{rl}
I(H)=&\left\{\iota(\eta)\in H^0(M',mK_{M'})\big|\ c(\eta_0,\eta)=0\right\}\\
\\
=&\left\{\iota(\eta)\in H^0(M',mK_{M'})\big|\ c(\iota\eta_0,\iota\eta)=0\right\}\\
\\
=&\left\{\eta'\in H^0(M',mK_{M'})\big|\ c(\iota\eta_0,\eta')=0\right\}\text{.}
\end{array}$$
By the first $\Longleftrightarrow$ in $(3.1)$, showing that $I(H)$ is in the image of $\varphi'=\varphi_{|mK_{M'}|}$
is equivalent to showing that $Bs|I(H)|-B'\neq\phi$, where $B'=Bs|mK_{M'}|$.

By Theorem 2.4.2, $C(D_{\iota\eta_0})\subseteq Bs|I(H)|$, hence it
suffices to prove $C(D_{\iota\eta_0})\nsubseteq B'$. Assume this
to be false, i.e.~$C(D_{\iota\eta_0})\subseteq B'$. Since $\iota$
is an isometry, $\iota\eta_0$ has the same asymptotic behavior as
that of $\eta_0$, hence $2{\rm
lct}(D_{\iota\eta_0})+\frac{2}{m}=2{\rm
lct}(D_{\eta_0})+\frac{2}{m}<1$. Theorem 2.4.2 then implies that
$\left\{\eta'\big|\
c(\iota\eta_0,\eta')=0\right\}=H^0(M',mK_{M'})$, which is also
$I(H)$ as shown in last paragraph, a contradiction.
\end{proof}
A more general image identifying result using the pseudonorms can be found in \cite{chi} and \cite{cy}.
\begin{rema}
The more detailed asymptotic expansions which are mentioned in
Remark 2.2.2 actually allow us to remove the condition $2{\rm
lct}(D_{\eta_0})+\frac{2}{m}<1$ in the definition of $\rm (CS)$.
This is useful in getting better uniform bounds for the results in
$\mathbf{4}$. See \cite{chi}.
\end{rema}

\section{Birational Equivalence between Smooth Varieties of General Type}
In this section $M$ will be a smooth compact complex manifold such
that the rational map $\varphi_{|mK_M|}$ maps $M$ to its image
birationally for sufficiently large $m$. We want to know for which
$r\in\mathbf{N}$ $\rm (CS)$ holds for $rK_M$.

In case $rK_M$ maps $M$ birationally to its image, the condition
$\rm (CS)$ admits an equivalent statement in terms of points in
$M$ instead of those in the image. It is clear that in this case
$\rm (CS)$ can be restated in the following way:\medskip

\noindent $\rm (CS)$ For a generic point $x$ in $M$ there exists
$\eta_0\in H^0(M, rK_M)$ such that
$$({\rm lct}(D_{\eta_0},x),\mu (\eta_0,x))>({\rm lct}(D_{\eta_0},y),\mu (\eta_0,y))$$ for any $y\neq x$
and ($2{\rm lct}(D_{\eta_0})+\frac{2}{m}=$) $2{\rm lct}(D_{\eta_0},x)+\frac{2}{m}<1$.
\begin{defi} $(i)$ For any $x\in M$ we define
$$V(r,x)=\left\{\eta\in H^0(M,rK_M)\big|\ {\rm mult}_x\eta\geq\frac{2nr}{r-2}\right\}\text{.}$$
$(ii)$
$$\mathcal S_M=\left\{3\leq r\in\mathbf{N}\left|
\begin{array}{l}
(i)\ Bs|rK_M|=\phi\text{, and}\\
(ii)\ \text{for a generic }x\in M\ Bs|V(r,x)|=\{x\}\text{.}\\
\end{array}
\right.\right\}$$
\end{defi}
\begin{rema}
$(ii)$ in particular implies that $\varphi=\varphi_{|rK_M|}$ maps
$M$ birationally to its image. Indeed for a generic $x\in M$
$$\begin{array}{l}\varphi^{-1}\varphi (x)=\varphi^{-1}(\{\eta\in H^0(M,rK_M)|\eta (x)=0\})\\
\\
=Bs\big|\{\eta\in H^0(M,rK_M)|\eta (x)=0\}\big|\subseteq Bs|V(r,x)|=\{x\}\text{.}
\end{array}$$
\end{rema}

\begin{lemm}
If $r\in\mathcal S_M$ then $\rm (CS)$ holds for $rK_M$.
\end{lemm}
\begin{proof}
$(ii)$ and Bertini's theorem imply that for a generic $x\in M$
there is $\eta_0\in H^0(M,rK_M)$ such that ${\rm
mult}_x\eta_0\geq\frac{2nr}{r-2}$ and ${\rm mult}_y\eta_0\leq 1$
for $y\neq x$. This shows that ${\rm lct}(D_{\eta_0},y)=1$ or
$\infty$ and by (2.3) that ${\rm
lct}(D_{\eta_0},x)\leq\frac{n}{{\rm
mult}_x\eta_0}<\frac{r-1}{2r}=\frac{1}{2}-\frac{1}{r}<\frac{1}{2}$.
It is clear from Definition 2.3.1(3) that $$({\rm
lct}(D_{\eta_0},x),\mu (\eta_0,x))>({\rm lct}(D_{\eta_0},y),\mu
(\eta_0,y))\text{.}$$ Besides, $2{\rm
lct}(D_{\eta_0},x)+\frac{2}{r}<\frac{2n(r-2)}{2nr}+\frac{2}{r}=1$.
\end{proof}

\begin{lemm}
$\mathcal{S}_M$ is a semigroup, i.e. $(r_1+r_2)\in\mathcal{S}_M$ if $r_1$, $r_2\in\mathcal{S}_M$.
\end{lemm}

\begin{proof}
Condition $(i)$ obviously holds for $(r_1+r_2)$ if it does for
$r_1$ and $r_2$. As for condition $(ii)$, for $x$ in some Zariski
open subset $U\subseteq M$ we have
$Bs|V(r_1,x)|=Bs|V(r_2,x)|=\{x\}$ since $r_1$ and $r_2\in\mathcal
S_M$. We want to show that for $x\in U$, $y\notin
Bs|V(r_1+r_2,x)|$ if $y\neq x$. By Bertini's theorem we can find
$\eta_j\in V(r_j,x)$ such that $\eta_j(y)\neq 0$ for $j=1,2$. Let
$\eta=\eta_1\otimes\eta_2\in H^0(M, (r_1+r_2)K_M)$. We have
$${\rm mult}_x\eta={\rm mult}_x\eta_1+{\rm mult}_x\eta_2>\frac{2nr_1}{r_1-2}+\frac{2nr_2}{r_2-2}>\frac{2n(r_1+r_2)}{r_1+r_2-2}$$
by the fact that $\frac{x+y}{x+y-2}<\frac{x}{x-2}+\frac{y}{y-2}$ if $x$, $y\geq 3$.
Therefore $\eta\in V(r_1+r_2,x)$ and $\eta(y)\neq 0$. So $y\notin Bs|V(r_1+r_2,x)|$.
\end{proof}
\begin{lemm}
Suppose $Bs|mK_M|=\phi$ and $\varphi_{|mK_M|}$ maps $M$ onto its
image in $\mathbb{P}H^0(M,mK_M)^*$ birationally. Then $\nu
m\in\mathcal S_M$ for any integer $\nu\geq\ 2n+1$.
\end{lemm}

\begin{proof}
Condition $(i)$ in the definition of $\mathcal S_M$ obviously holds.
Only $(ii)$ needs to be verified.

Since $\varphi_{|mK_M|}$ maps $M$ to its image birationally, we
can find Zariski open subsets $U_0$ and $U$ of $M$ and the image
of $\varphi_{|mK_M|}$ respectively such that
$\varphi_{|mK_M|}:U_0=\varphi^{-1}_{|mK_M|}(U)\mathop{\to}\limits^\sim
U$.

We want to show that $y\notin Bs|V(\nu m,x)|$ if $y\neq x$ (i.e.
$(ii)$) for $x\in U_0$. By the choice of $U_0$ it is clear that
$Bs\big|\{\eta\in H^0(M,mK_M)|\ \eta (x)=0\}\big|=\{x\}$.
Therefore, for any $y\neq x$ there exists $\eta\in H^0(M,mK_M)$
such that $\eta (x)=0$ and $\eta (y)\neq 0$. Taking
$\eta_0=\eta^{\otimes\nu}\in H^0(M,\nu mK_M)$, we have $\eta\in
V(\nu m,x)$ since $${\rm mult}_x\eta_0=\nu{\rm
mult}_x\eta\geq\nu>\frac{2n\nu m}{\nu m-2}$$ when $\nu >2n+1$. So
$y\notin Bs|V(\nu m,x)|$.
\end{proof}
\begin{lemm} Let $M$ be a nonsingular complex projective variety of general type and of dimension $n$. Let $d\in\mathbf{N}$
be such that $Bs|mdK_M|=\phi$
for $m\geq m_0$. Then there exists
$r_0\in\mathbf{N}$ depending only on $m_0$ and $n$ such that $rd\in\mathcal S_M$ if $r\geq r_0$.
\end{lemm}

\begin{proof}
It is proved in \cite{hm} and \cite{takayama} that for each
$n\in\mathbf{N}$ there exists $m_n\in\mathbf{N}$ such that if $M$
is a smooth projective variety of general type and of dimension
$n$ then the rational map $\varphi_{|mK_M|}$ maps $M$ to its image
birationally for any $m\geq m_n$.

Choose distinct prime numbers $m$, $m'$, $\nu$ and $\nu'$ such that $m, m'\geq\max\{\frac{m_n}{d}, m_0\}$,
$\nu$ and $\nu'>2n+1$. Then Lemma 4.3 implies that
$m\nu d$ and $m'\nu'd\in\mathcal S_M$ and Lemma 4.2 implies the lemma.
\end{proof}
Our main theorem is the following
\begin{theorem}
Let $M$ and $M'$ be smooth complex projective varieties of general
type and of dimension $n$ and $d\in\mathbf{N}$ such that
$Bs|mdK_M|= Bs|mdK_{M'}|=\phi$ for $m\geq m_0$. Let
$r_0\in\mathbf{N}$ as given by Lemma 4.4.

If for some $r\geq r_0$ we have a linear isometry
$$\iota :\big(H^0(M, rdK_M),\langle\langle\ \rangle\rangle\big)\to\big(H^0(M', rdK_{M'}),\langle\langle\ \rangle\rangle\big)$$
then there exists a unique birational map $\psi:M\dashrightarrow M'$ and $c\in\mathbf{C}$ with $|c|=1$ such that
$c\iota=\psi^*$, the isomorphism induced by $\psi$.
\end{theorem}
\begin{proof}
Lemma 4.1 and Lemma 4.4 together imply the $\rm (CS)$ holds for
$\rho dK_M$ and $\rho dK_{M'}$ if $\rho\geq r_0$. By Lemma 4.4 and
Remark 4.2 $\varphi_{|rdK_M|}$ and $\varphi_{|rdK_{M'}|}$ map $M$
and $M'$ birationally to their images respectively. Denote the
isomorphism induced by $\iota$ as $$I:\mathbb{P}H^0(M,
rdK_M)^*\to\mathbb{P}H^0(M', rdK_{M'})^*\text{.}$$ The assumption
and Lemma 3.1 implies that $I$ identifies the images of
$\varphi_{|rdK_M|}$ and $\varphi_{|rdK_M|}$. Therefore we obtain a
unique birational map $\psi$ making the following diagram of
rational maps commutative:
$$\xymatrix{M\ar@{.>}[r]^{\psi} \ar[d]_{\phi_{|rdK_M|}} & M'\ar[d]^{\phi_{|rdK_{M'}|}}\\
\mathbb{P}H^0(M,rdK_M)^*\ar[r]^I & \mathbb{P}H^0(M',rdK_{M'})^*}$$

Let $\psi^*: H^0(M, rdK_M)\to H^0(M', rdK_{M'})$ be the isomorphism induced by $\psi$. It is an isometry with respect to
$\langle\langle\ \rangle\rangle_{rd}$. Since $\psi^*$ and $\iota$ both induce
$I:\mathbb{P}H^0(M, rdK_M)^*\to\mathbb{P}H^0(M', rdK_{M'})^*$, there is $c\in\mathbf{C}$ such
that
$c\iota=\psi^*$. Both $\iota$ and $\psi^*$ are isometries with respect to those $\langle\langle\ \rangle\rangle$s,
hence $|c|=1$.
\end{proof}

Using this theorem we can obtain several uniform results. For example, in the case $n=2$, we can even have $r_0$
depending only on $n$. The reason is that it is enough to prove the theorem for $M$
and $M'$ both minimal models. By the classical results due to Bombieri and Kodaira $Bs|mK_M|=Bs|mK_{M'}|=\phi$ if
$m\geq 5$. The proof of Lemma 4.4 shows that $\mathcal{S}$, the additive semigroup of $\mathbf{N}$ generated by $\{ab\ |\ a,b\in\mathbf{N}, a\geq 5, b\geq 6\}$, is contained in $\mathcal{S}_M$. It is not hard to see that $m\in\mathcal S$ for any $m\geq 75$, and hence $r_0$ can be chosen to be 75. Then we can take $d=1$ and $m_0=5$ in Theorem 4.1 and get the following

\begin{theorem}
Given a linear isometry $$\iota:\big(H^0(mK_M),\langle\langle\ \rangle\rangle\big)\to \big(H^0(mK_{M'}),\langle\langle\ \rangle\rangle\big)$$ for some $m\geq 75$, there exists a unique pair of a birational map $\psi:M'\dashrightarrow M$ and a complex number $c$ of unit length such that $\psi^*$, the isomorphism induced by $\psi$, is equal to $c\iota$.
\end{theorem}

For higher dimensions, in the same spirit we obtain the following
\begin{theorem}
There exists $r_0\in\mathbf{N}$ which depends on $n$, such that for any two smooth
complex projective varieties $M$ and $M'$ of general type and of
dimension $n$ which both admit smooth minimal models, if for some $r\geq
r_0$ we have a linear isometry
$$\iota :\big(H^0(M, 2r(n+2)!K_M),\langle\langle\ \rangle\rangle\big)\to\big(H^0(M', 2r(n+2)!K_{M'}),\langle\langle\ \rangle\rangle\big)$$
then there exists a unique birational map $\psi:M\dashrightarrow
M'$ and a unique complex number $c$ of unit length such that $\psi^*$, the
isomorphism induced by $\psi$, is equal to $c\iota$ .
\end{theorem}

\begin{proof}
As remarked in the paragraph before Theorem 4.2 we may assume that
$M$ and $M'$ are both minimal models, i.e. $K_M$ and $K_{M'}$ are
both nef.

Koll\'ar's effective base freeness theorem (\cite{kollar1},
$\mathbf{1.1\ Theorem}$) says that if a log pair $(X,\Delta)$ is
proper and klt of dimension $n$, $L$ a nef Cartier divisor on $X$,
and $a\in\mathbf{N}$ such that $aL-(K_X+\Delta)$ is nef and big,
then $|2(n+2)!(a+n)L|$ is base point free. Applying this to the
case $X=M$ (resp.~$M'$), $\Delta=0$, $L=K_M$ (resp.~$K_{M'}$) and
$a\geq 2$ we have that $Bs|2m(n+2)!K_M|=Bs|2m(n+2)!K_{M'}|=\phi$
if $m\geq n+2$.

Therefore we may take $d=2(n+2)!$ and $m_0=n+2$ in Lemma 4.4 and Theorem 4.1, and then the theorem follows.
\end{proof}

\begin{rema}
It is shown in \cite{bchm} that every variety of general type
admits a minimal model. However in the proof above the smoothness
of the minimal models are required.
\end{rema}

Here, in order to illustrate how the main idea goes, we only deal with the case when suitable base
point free conditions hold. The presence of base loci is another technical issue. By a careful analysis and modification of the results in $\mathbf{2}$, a suitable use of the effective base point
freeness, and the existence of minimal models for varieties of general type, we are still able to say something for the general case. The following theorems 4.4 and 4.5 are the precise results whose proofs can be found in \cite{chi} and \cite{cy}.

We first recall some facts about the minimal models. It is known that
every projective manifold $X$ of general type admits a minimal model $Y$ with $K_Y$ $\mathbf Q$-Cartier\cite{bchm}. The index of
$Y$ is defined as $j_Y=\min\{j\ |\ jK_Y\text{ is Cartier}\}$. It is also known that any two birational minimal models have the same index. Hence we can define the index of a projective manifold to be that of any of its minimal models. We have the following
\begin{theorem}{\rm (}\cite{chi} and \cite{cy}{\rm )}
For every natural number $j$ there exists $r_{n,j}$ which depends only on $n$ and $j$ such that
given any two $n$-dimensional projective manifolds $M$ and $M'$ with indices $j$, and a linear isometry $$\iota:\big(H^0(M,2r(n+2)!K_M),\langle\langle\ \rangle\rangle\big)\to \big(H^0(M',2r(n+2)!K_{M'}),\langle\langle\ \rangle\rangle\big)$$ for some $r\geq r_{n,j}$, there
exists a unique birational map $\psi:M'\dashrightarrow M$ and a unique complex number $c$ of unit length such that the induced map $\psi^*$ is equal to $c\iota$.
\end{theorem}

The number $r_{n,j}$ in this theorem depends not only on the dimension $n$ but also on the index of minimal models.
To get a uniform result in higher dimensional cases, we need to introduce some objects here. Let
$$V(M,m,r)={\rm image}\left({\rm Sym}^r H^0(M, mK_M)\to H^0(M,rmK_M)\right)$$ for any $m,r\in\mathbf N$, where the map
is the canonical one. $V(M,m,r)$ inherits from $\left(H^0(M,rmK_M),\langle\langle\ \rangle\rangle_{rm}\right)$ a pseudonorm, still denoted as $\langle\langle\ \rangle\rangle_{rm}$. It is clear that $\big(V(M,m,r),\langle\langle\ \rangle\rangle_{rm}\big)$ is a birational invariant.\\

Recall also the definition of $m_n$ in the proof of Lemma 4.4, which is a number such that $\phi_{|mK_M|}$ maps
$M$ birationally to its image for every $m\geq m_n$. With these notions, we can also prove the following result :

\begin{theorem}{\rm (}\cite{chi} and \cite{cy}{\rm )}
Given a linear isometry $$\iota :\big(V(M,m,r),\langle\langle\ \rangle\rangle\big)\to\big(V(M',m,r),\langle\langle\ \rangle\rangle\big)$$ for some $r\geq 2n+1$ and $m\geq m_n$, there
exists a unique birational map $\psi:M'\dashrightarrow M$ and a unique complex number $c$ of unit length such that the induced map $\psi^*$ is equal to $c\iota$.
\end{theorem}

\begin{rema}
Many of the results in this paper have a more general version for
$L+mK_M$ where $L$ is a hermitian line bundle (see \cite{chi}).
\end{rema}


\vspace{1cm}
\begin{tabular}{ccc}
\begin{minipage}[t]{\tw}
\textbf{C.-Y. Chi}\\
Department of Mathematics\\
Harvard University\\
1 Oxford street\\
Cambridge, MA 02138\\
email: cychi@math.harvard.edu
\end{minipage}
&
\begin{minipage}[t]{\tw}
\textbf{S.-T. Yau}\\
Department of Mathematics\\
Harvard University\\
1 Oxford street\\
Cambridge, MA 02138\\
email: yau@math.harvard.edu
\end{minipage}
\end{tabular}\end{document}